\documentclass[11pt,a4paper, reqno]{amsart}
\usepackage{mathrsfs}

\usepackage{amsmath,amsfonts,verbatim}
\usepackage{latexsym}
\usepackage{amssymb,leftidx}
\usepackage{extarrows}
\usepackage{overpic}
\usepackage{color}
\usepackage{epsfig}
\usepackage{subfigure}
\usepackage{tikz}
\usepackage{hyperref}
\usepackage{enumitem}
\colorlet{RED}{red}

\hypersetup{urlcolor=blue, citecolor=red}

\usepackage{amssymb}
\usepackage{mathrsfs}
\usepackage{amscd}
\usepackage{bbm}
\usepackage{cite}

\newcommand\R{\mathbb{R}}

\newcommand{\norm}[1]{\left\lVert #1\right\rVert}

\newcommand\A{\bf A}

\usepackage{graphicx}
\usepackage{float}

\numberwithin{equation}{section}

\newtheorem{definition}{Definition}[section]
\newtheorem{lemma}{Lemma}[section]
\newtheorem{theorem}{Theorem}[section]

\begin{document}
\title[Carleman estimate and its applications in electromagnetic fields]{Carleman Estimates for Critical Electromagnetic Schr\"odinger Operators on Conic Manifolds}

\author{Baiyang Yin}
\address{Department of Mathematics, Beijing Institute of Technology, Beijing 100081; }
\email{baiyang\_yin@bit.edu.cn}

\author{Fang Zhang}
\address{School of Mathematics and Information Science, Henan Polytechnic University, Jiaozuo, Henan, China, 454003;}
\email{zhangfang@hpu.edu.cn}

\author{Junyong Zhang}
\address{Department of Mathematics, Beijing Institute of Technology, Beijing 100081; }
\email{zhang\_junyong@bit.edu.cn; }

\begin{abstract}

We establish an $L^p$--$L^q$-type Carleman estimate for the Schr\"odinger operator $H_{A,V}$ with scaling-critical electromagnetic singular potentials on a conical singular space $(X,g)$ of dimension $d\geq3$, where the metric is $g=dr^2+r^2h$ and $X=C(Y)=(0,\infty)_r\times Y$ is a product cone over the closed Riemannian manifold $(Y,h)$. More precisely, we prove
\[
\|e^{\tau\phi}u\|_{L^{2d/(d-2)}(X)}\le C\,\|e^{\tau\phi}H_{A,V}u\|_{L^{2d/(d+2)}(X)}
\]
with the quadratic weight $\phi(r)=\tfrac14(\log r)^2$, uniformly in $\tau$.
As an application of this Carleman estimate, we prove the unique continuation property for this electromagnetic Schr\"odinger operator when the cross-section is $Y=\mathbb{S}^{d-1}$.

\end{abstract}

\maketitle

\begin{center}
 \begin{minipage}{120mm}
 \end{minipage}
 \end{center}


\tableofcontents

\section{Introduction and main results}
\label{sec:intro}
\subsection{The Background and Motivation} Let $H$ be a self-adjoint operator. The
Carleman estimates are weighted a priori inequalities of the form
\begin{equation}\label{equ:carleman}
\bigl\|e^{\tau\phi} u\bigr\|_{L^{q}} \le C \bigl\|e^{\tau\phi} H u\bigr\|_{L^p},
\end{equation}
where $\phi$ is a weight function and $\tau$ is an arbitrary parameter. A key feature is that the parameter $\tau$ may be taken arbitrarily large, provided that $\tau\ge \tau^*$. Consequently, regions where $\phi$ is large dominate the estimate and impose constraints on the behaviour of $u$ elsewhere. Such estimates were introduced by T.\,Carleman \cite{Carleman39} in 1939 to prove uniqueness of the Cauchy problem for elliptic systems in two independent variables. The underlying idea---that an exponential weight forces uniqueness to propagate across a hypersurface---has become one of the fundamental tools in partial differential equations, with applications ranging from unique continuation to control theory and inverse problems.
For applications to unique continuation, we refer to \cite{Hormander83, JK85, KRS87, KT01}; for applications of Carleman estimates to control theory, see \cite{FI96, LR95, FCG06, Lions88}, and for applications to inverse problems, see \cite{BK81, Klib92, IY98, Yam09, BY17}.

The classical $L^2$ theory was developed by A.-P.\,Calder\'on \cite{Calderon58} and L.\,H\"ormander \cite{Hormander63,Hormander83}. Calder\'on \cite{Calderon58} introduced the weight convexity condition and established uniqueness for the Cauchy problem for linear operators with non-vanishing principal symbol. H\"ormander \cite{Hormander63} reformulated the theory within the pseudodifferential and commutator framework and isolated the sharp pseudoconvexity condition guaranteeing an $L^2$ Carleman estimate. In that setting, the weight is only required to be non-characteristic in an appropriate sense; the resulting estimates are robust under lower-order perturbations, but they are generally not scaling-invariant. Most of the above results apply to operators with smooth coefficients, while Koch--Tataru \cite{KT01} derived $L^2$ Carleman estimates admitting non-smooth coefficients. $L^p$-type Carleman estimates for more general operators and their applications to unique continuation were obtained by Ionescu--Kenig \cite{IK04} and Dos Santos Ferreira \cite{DF05}; related $L^p$--$L^q$ uniform resolvent estimates for Schr\"odinger operators were further investigated by Huang--Sogge \cite{HS21}.

A second, more restrictive line of research pursues $L^p$--$L^q$ Carleman estimates tailored to operators with homogeneous critical singularities. Landmark results in this direction are the uniform Sobolev inequalities of Kenig--Ruiz--Sogge \cite{KRS87},
and the weighted estimates of Jerison--Kenig \cite{JK85} (see also Jerison \cite{Jerison86}). They proved that for all $f\in C_c^\infty(\mathbb R^d\setminus\{0\})$,
\begin{equation}
\big\||x|^{-\tau}f\big\|_{L^q(\R^d)}\le C_\delta \big\||x|^{-\tau}\Delta f\big\|_{L^p(\R^d)},
\label{eq:JK}
\end{equation}
whenever $\operatorname{dist}\bigl(\tau,\mathbb Z+\tfrac{d-2}{2}\bigr)\ge\delta>0$, with endpoint exponents $q=\frac{2d}{d-2}$, $p=\frac{2d}{d+2}$. It is worth noting that a special spectral-gap property is used in the argument. Considering $-\Delta$ in polar coordinates, the eigenvalues of the conformal Laplacian $-\Delta_{\mathbb{S}^{d-1}}+\bigl(\frac{d-2}{2}\bigr)^2$ are given by $\lambda_k^2=\bigl(k+\frac{d-2}{2}\bigr)^2$, which gives a key gap $|\lambda_{k+1}-\lambda_k|\geq 1$. We refer to Sogge \cite{sogge2017fourier} for details. 

For Schr\"odinger operators with singular magnetic and electric potentials, unique continuation for magnetic fields was first proved by Kurata \cite{Kur97}. For the scaling-critical regime $A\sim|x|^{-1}$ and $V\sim|x|^{-2}$, Arrizabalaga--Zubeldia \cite{AZ15} obtained \(L^2\)-type Carleman estimates on \(\mathbb R^d\) for operators of the form
$H_{A,V}=\nabla_A^2+V$ with 
 $\nabla_A:=\nabla+iA$,
and established unique continuation from infinity. In this paper, we study $L^p$--$L^q$-type Carleman estimates for the Schr\"odinger operator $H_{A,V}$ with scaling-critical electromagnetic singular potentials. The spectral gap property mentioned above may break down in the presence of scaling-critical singular electromagnetic potentials.

\subsection{The setup} In this paper, we consider the Schr\"odinger operator $H_{A,V}$ with scaling-critical electromagnetic singular potentials on a conical singular space $(X,g)$. By a conic manifold we mean
\begin{equation}
X=C(Y)=(0,\infty) \times Y,
\end{equation}
which is a $d$-dimensional product cone over the closed Riemannian manifold $(Y, h)$. $X$ is equipped with the metric
\begin{equation}
g=dr^2+r^2 h.
\end{equation}
We will in fact view it as its completion\footnote{One can check that this `closure' is indeed complete.} $\overline{X}=[0,\infty) \times Y$. The Euclidean space $\R^d$ is the simplest example of a cone, with the cross-section $Y=\mathbb{S}^{d-1}$ with its standard metric.
While the general product cones $X$ possess a dilation symmetry analogous to that of Euclidean space, they have no other symmetries in general.

Denoting the dual metric of $g$ by $\big(g^{jk}\big)$ and the determinant by $|g|=\mathrm{det}\big(g_{jk}\big)$, the electromagnetic Schr\"odinger operator $H_{A,V}$ on the product cone $X=C(Y)$ is given by
\begin{equation}\label{H-AV}
\begin{split}
H_{A, V}&=\frac1{\sqrt{|g|}}\sum_{j,k=1}^d\Big(i\frac{\partial}{\partial x_j}+A_j\Big)\sqrt{|g|} g^{jk} \Big(i\frac{\partial}{\partial x_k}+A_k\Big)+V(x)\\
&=-\Delta_g+2iA\cdot\nabla_g+i\mathrm{div}_g A+|A|^2+V(x)
\end{split}
\end{equation}
where $\cdot$ is the pairing of 1-forms and vector fields, $V: X\to \R$ is the electric scalar potential and 
\begin{equation} \label{eq:defn-A}
A = A_jdx^j
\end{equation}
is the 1-form representing the magnetic vector potential, and $|A|^2$ is defined by $\sum_{j,k=1}^{d}\\g^{jk}A_jA_k$. 
When we take the divergence of 1-forms, we always mean the divergence of the vector field obtained from the 1-form using the metric.

In this paper, we focus in particular on the Schr\"odinger operator of the form \eqref{H-AV}, featuring Coulomb-type magnetic potentials and inverse-square electric potentials, that is
\begin{equation} \label{eq:A-form}
A = {\A}(y), \quad V = \frac{a(y)}{r^2},\quad y \in Y,
\end{equation}
where $\A$ is a 1-form on $Y$ and the first equality is interpreted as viewing ${\A}$ as a form on $X$ that is invariant under translations in $r$.
This imposes the condition that $A$ has no radial component: $g(A,dr)=0$, which is the so-called Cronstr\"om or transversal gauge. 
We may think of $\eqref{eq:A-form}$ as requiring 
\begin{equation}
A \sim O(r^{-1}), \quad r\to 0, +\infty.
\end{equation}
This is because the components above are (implicitly) expressed in terms of the basis of 1-forms $dx^j$, which is equivalent to $dr$ and $r\,d\hat{x}$, where $\hat{x}$ denotes a coordinate system on $Y$. From \eqref{eq:A-form} and \eqref{H-AV}, we obtain
\begin{equation}\label{eq:LAa-general}
\begin{split}
H_{A, V}&=-\partial_r^2-\frac{d-1}r\partial_r+\frac1{\sqrt{|h|}}\sum_{j,k=2}^d\Big(i\frac{\partial}{\partial \hat{x}_j}+A_j\Big)\sqrt{|h|} h^{jk} \Big(i\frac{\partial}{\partial \hat{x}_k}+A_k\Big)+\frac{a(\hat{x})}{r^2}.
\end{split}
\end{equation}
Thus, in our setting the operator in \eqref{eq:LAa-general} becomes
\begin{equation}\label{eq:LAa-general'}
\begin{split}
H_{A, V}&=-\Delta_g+\frac{|{\A}(\hat{x})|^2+i\,\mathrm{div}_{h}{\A}(\hat{x})+a(\hat{x})}{r^2}+2i {\A}(\hat{x})\cdot \nabla_g\\
&=-\partial_r^2-\frac{n-1}r\partial_r+ \frac{L_{{\A},a}}{r^2},
\end{split}
\end{equation}
where $\hat{x}\in Y$, $|{\A}(\hat{x})|^2 = \sum_{i,j=1}^{d-1} h^{jk}{\A}_j{\A}_k$, and
\begin{equation}\label{L-angle}
\begin{split}
L_{{\A},a}
=-\Delta_{h}+\big(|{\A}(\hat{x})|^2+i\,\mathrm{div}_{h}{\A}(\hat{x})+a(\hat{x})\big)+2i {\A}(\hat{x})\cdot\nabla_{h}.
\end{split}
\end{equation}
In the special case of the cross-section $Y=\mathbb{S}^{d-1}$, we have
\begin{equation}\label{eq:LAa-Euclidean}
H_{A, V}=\Big(i\nabla+\frac{{\A}(\hat{x})}{|x|}\Big)^2+\frac{a(\hat{x})}{|x|^2},
\end{equation}
which is the exact Schr\"odinger operator studied in \cite{Fanelli2015, Fanelli2013, Felli2011, Fanelli2022} for dispersive and Strichartz estimates. The Euclidean Schr\"odinger operator \eqref{eq:LAa-Euclidean}, which includes the critical Coulomb-type decay potential, has garnered significant attention from both the mathematical and physical communities. For instance, the Aharonov-Bohm effect \cite{AharonovBohm1959, EhrenbergSiday1949} arises from the critical Coulomb-type decay potential, and the diffractive behavior of waves in such potentials has been investigated in \cite{Yang2021,Yang2022}. The asymptotic behavior of time-independent Schr\"odinger solutions was analyzed in \cite{Felli2011}.

 \vspace{0.2cm}

\subsection{The main results}
Let $\Delta_h$ and $\nabla_h$ denote the Laplace--Beltrami operator and the Riemannian gradient on $Y$, respectively. We impose the following assumptions on the potentials:
\begin{enumerate}[label=\textup{(A\arabic*)}, ref=\textup{(A\arabic*)}]
    \item\label{assump:A1} $a \in C^{\infty}(Y)$, $\mathbf{A}\in C^{\infty}(Y)$
    \item\label{assump:A2} The operator
    $\bigl(i\nabla_{h} + \mathbf{A}\bigr)^2 + a(y) + \bigl(\tfrac{d-2}{2}\bigr)^2$
    is positive on $Y$, i.e., there exists $\delta > 0$ such that
    \begin{equation}
    \int_{Y}
    \Bigl(\big|\bigl(i\nabla_{h} + \mathbf{A}\bigr)u\big|^2
    + \Bigl(a(y) + \bigl(\tfrac{d-2}{2}\bigr)^2\Bigr)|u|^2 \Bigr)\,\mathrm{d} h
    \geq \delta \int_{ Y} |u|^2 \,\mathrm{d} h.
    \end{equation}
\end{enumerate}

\begin{theorem}\label{thm:main}
Assume that \(a(y)\) and $\A$ are given by \eqref{eq:A-form} and satisfy (A1) and (A2). Let \(d\ge 3\), \(\tau>0\), and
\(\phi(r,y) = \frac14 \bigl(\log r\bigr)^2\). Then there exists a constant \(C\),
depending on \(d\), \(\mathbf{A}\), \(a\) and \(\phi\), but not on \(\tau\) or \(u\),
such that for all \(u\in C^\infty_c(X)\) and all sufficiently large $\tau$,
\begin{equation}\label{est:main}
\| e^{\tau \phi} u \|_{L^{\frac{2d}{d-2}}(X)}
\le C \| e^{\tau \phi} H_{A,V} u \|_{L^{\frac{2d}{d+2}}(X)}.
\end{equation}
\end{theorem}
 \vspace{0.2cm}

As an application of Theorem \ref{thm:main}, we study the unique continuation property of $H_{A,V}$ in the Euclidean case, that is, the section cross $Y=\mathbb{S}^{d-1}$. We first recall the definition of unique continuation. Consider a partial differential operator $P(x, D)$ of order $m$ in $\mathbb R^d$.
Let $\Sigma$ be an oriented hypersurface in $\mathbb R^d$ that can be represented as a nondegenerate level set $\Sigma=\{\psi = 0\}$ of a smooth function $\psi$.
The sign of $\psi$ away from $\Sigma$ determines the orientation of $\Sigma$.
Denote the two sides of $\Sigma$ by $\Sigma^+ =\{\psi> 0\}$ and $\Sigma^- =\{\psi< 0\}$.
Then we define the unique continuation property across $\Sigma$ for solutions to $P(x, D)u = 0$ as follows:

\begin{definition}
We say that the unique continuation property across $\Sigma$ holds for the operator $P(x, D)$ if for each $x_0 \in \Sigma$ there exists a neighborhood $V$ of $x_0$ such that the following holds: Let $u$ be a solution to $P(x, D)u = 0$ in $V$ such that $u = 0$ in $\Sigma^+ \cap V$. Then $u = 0$ near $x_0$.
\end{definition}

The electromagnetic operator $H_{A,V}$ has the unique continuation property in sense that:
\begin{theorem}
Let $D_{a,b} = \{a<|x|<b\}$, and let $f\in W^{2,p}_{\mathrm{loc}}(\mathbb R^d)$ solve $H_{A,V}f=0$ in $\mathbb R^d$ and satisfy assumptions \ref{assump:A1}--\ref{assump:A2}. If $f=0$ on the shell $a<|x|<b$, then $f=0$.
 \label{prob:5.2}
\end{theorem}

In addition, we may consider the so-called \textit{strong unique continuation property} (SUCP).
Roughly speaking, an operator $P$ is said to possess the SUCP if any solution $u$ of $Pu=0$ that vanishes of infinite order at some point $x_0$ must vanish identically in the connected component containing $x_0$.
The Carleman estimate approach is one of the most powerful tools to establish the strong unique continuation property for second-order differential operators.

\begin{theorem}
    Let $f\in W^{2,p}_{\mathrm{loc}}(\mathbb R^d)$ be a solution to $H_{A,V} f=0$ in $\mathbb R^d$, and suppose that the assumptions \ref{assump:A1}--\ref{assump:A2} hold.
Assume that $f$ satisfies the following vanishing condition at the origin:
\[
\int_{|x|<R}\big(|f|^p+|\nabla_A f|^p\big)\mathrm dx
= O\big(e^{-N(\log R)^2}\big)\qquad \text{for all } N>0.
\]where $\nabla_A:=i\nabla+A$.
Then $f=0$ almost everywhere in $\mathbb R^d$.
\label{theorem1.3}
\end{theorem}

\vspace{0.2cm}

{\bf Organization.}The outline of this paper is as follows.
Section \ref{sec:proof} analyzes how the choice of weight affects the invertibility of the operator.
We show that a pointwise bounded inverse is unlikely to exist, then construct an appropriate bounded left-inverse and derive its estimates.
These estimates are subsequently used to prove the main theorem.
In Section \ref{sec:uc_application}, we apply our Carleman estimate to the unique continuation problem and prove Theorem \ref{prob:5.2} and Theorem \ref{theorem1.3}.
A scaled Carleman estimate is established, which enables unique continuation across the unit sphere.

\section{Proof of Theorem \ref{thm:main}}\label{sec:proof}

In this section we prove the main theorem. The main tool is Sogge's cluster estimates on section cross $Y$.

\subsection{Reduction to the conjugated operator}We recall \eqref{eq:LAa-general'} 
\begin{equation*}
\begin{aligned}
H_{A,V}
&= -\partial_r^2-\frac{d-1}{r}\partial_r+\frac{L_{\mathbf{A},a}}{r^2},
\end{aligned}
\label{eq:1.13}
\end{equation*}
where \(\hat x\in Y\), \(|\mathbf{A}(\hat x)|^2=\sum_{i,j=1}^{d-1}h^{jk}\mathbf A_j \mathbf A_k\), and
\begin{equation}
L_{\mathbf{A},a}= -\Delta_h+\bigl(|\mathbf{A}(\hat x)|^2+i\operatorname{div}_h\mathbf{A}(\hat x)+a(\hat x)\bigr)+2i\mathbf{A}(\hat x)\cdot\nabla_h.
\label{eq:1.14}
\end{equation}
By making the change of variables
$r = e^{t}$, we obtain
\begin{equation}
H_{A,V}=e^{-2t}\Bigl(-(\partial_t)^2-(d-2)\partial_t+L_{\mathbf{A},a}\Bigr).
\end{equation}
To prove \eqref{est:main} in Theorem \ref{thm:main}, it suffices to show
\begin{equation}\label{est:main'}
\bigl\|e^{\tau t^2/4}v\bigr\|_{L^{q}(dtdh)}\le C\bigl\|e^{\tau t^2/4}Pv\bigr\|_{L^{p}(dtdh)},
\end{equation}
where
\begin{equation}\label{def:qp-P}
q=\frac{2d}{d-2},\quad p=\frac{2d}{d+2},\quad
P=-(\partial_t)^2+L_{\mathbf{A},a}+\Bigl(\frac{d-2}{2}\Bigr)^2.
\end{equation}

\vspace{2em}
\subsection{Left-inverse of operators in the distributional sense}

To prove \eqref{est:main'}, we set $h(t) =\tau t^2/4 = \beta t^2/2$, $\beta =\tau/2$ and define
\[
w = e^{h(t)}v,\quad F = e^{h(t)} P v,
\]
then we need to prove
\begin{equation}
\|w\|_{L^q(\mathrm{d}t\,\mathrm{d}h)}
\lesssim \|P_t w\|_{L^p(\mathrm{d}t\,\mathrm{d}h)}.
\label{eq:P_t-estimate}
\end{equation}
where the conjugated operator
\begin{equation}\label{def:Pt}
P_t = e^{\tau\frac{t^2}{4}} P e^{-\tau\frac{t^2}{4}}.
\end{equation}

A natural question is whether $P_t$ admits a well-defined inverse operator $P_t^{-1}$. If it exists, the problem can simply be reduced to whether $P_t^{-1}$ is $L^q$--$L^p$ bounded. 
 To illustrate the influence exerted by the weight function and the operator structure on the symbol, we first fix some notation.

Let $$T=T_{{\A},a}:=\bigl(i\nabla_h+\mathbf{A}\bigr)^2+a+\Bigl(\frac{d-2}{2}\Bigr)^2$$ be an elliptic self-adjoint operator on $Y$.
Since $Y$ is a compact manifold, $T$ admits discrete real eigenvalues $\{\lambda_{k}^2\}_{k\ge 0}$ and corresponding eigenfunctions $h_{k,\ell}$ in $L^2(Y)$, where $0\leq \ell\leq d(k)$ with multiplicity $d(k)$.
The corresponding eigenspaces are defined by
\begin{equation}\label{Hk}
H_k \overset{\mathrm{def}}{=}
\text{span}_{0\leq \ell\leq d(k)}\{ h_{k,\ell}\}.
\end{equation}
and the projector 
\begin{equation}\label{def:Hk}
\mathcal{H}_k: L^2(Y)\to H_k(Y).
\end{equation}

We now explicitly write  the symbol of $P_t$ restricted to $H_{k}$.
\begin{equation}
\eta^2 + 2i\tau (\partial_t\phi)\eta + \tau \partial_t^2\phi + \lambda_{k}^2 - \tau^2(\partial_t\phi)^2.
\end{equation}
In particular, $Y=\mathbb{S}^{d-1}$ $d\ge 3$. Let $\Delta_{\mathbb{S}^{d-1}}$ denote the usual Laplacian on $\mathbb{S}^{d-1}$, then the eigenvalues of the conformal Laplacian $-\Delta_{\mathbb{S}^{d-1}}+\bigl(\frac{d-2}{2}\bigr)^2$
are given by $\bigl(k+\frac{d-2}{2}\bigr)^2$. Choosing a linear weight function yields the following operator symbol:
\begin{equation}
\eta^2 + 2i\tau \eta +  \lambda_k^2 - \tau^2,
\label{eq:symbol_laplacian_linear_weight}
\end{equation}
where
$\lambda_k^2=\big(k+\frac{d-2}{2}\big)^2$. As discussed in Section~\ref{sec:intro}, the validity of the Carleman estimates is investigated under the assumption $\operatorname{dist}\Bigl(\tau,\mathbb Z+\tfrac{d-2}{2}\Bigr)\ge\delta>0$.
Otherwise, resonance phenomena arise and the operator does not possess a well-defined inverse. This assumption is reasonable in our present setting, since the eigenvalues of the Laplacian on the unit sphere are mutually separated. In other words, one can choose a suitable $\tau$ to avoid all resonance points. Unfortunately, this property depends heavily on the operator, the underlying manifold, and the choice of weight function. A similar argument can be found in Section 4 of \cite{KT01}.
More precisely, if we add a perturbation to the operator, the originally separated spectrum on the unit sphere may spread out, and it may become impossible to avoid all resonance points regardless of the choice of $\tau$. We also point out, however, that under certain special circumstances, the influence of perturbations on the spectral distribution of the operator can also be quantified. If $Y=\mathbb{S}^{d-1}$, for perturbations of the spherical Laplacian by an odd function, it is shown in \cite{Guillemin78} that
\begin{equation}\label{eq:1.3}
\bigl|\lambda_{k}^2 - \bigl(k + (d-2)/2\bigr)^2\bigr| = O\bigl(1/k^2\bigr).
\end{equation}
This describes the asymptotic behaviour of the spectrum and shows that such a perturbation becomes negligible for sufficiently large $k$. In this case, one may still select an appropriate $\tau$ avoiding all resonances. Nevertheless, for first-order perturbations or for a general manifold $Y$, a careful study of the spectral distribution is required, which constitutes a delicate problem.

It is pointed out in \cite{KT01} that one may resort to a convex function as the weight function to tackle this problem. In what follows, we shall prove that \eqref{est:main} holds for our weight function
\[
\phi(x) = \frac14 \bigl(\log|x|\bigr)^2.
\]

Here we claim that, with this choice of weight function, $P_t$ possesses a well-defined left inverse operator. The lemmas below provide theire explicit form together with the corresponding estimates.

\begin{lemma}\label{lem:kernel}
Let \(\beta \ge 1\), \(m > 0\), \(h_\beta(t)=\beta t^2/2\), and
\[
Q_{\beta,m}(t)=e^{h_\beta(t)}\bigl(-\partial_t^2 + m^2\bigr)e^{-h_\beta(t)}.
\]
For every \(w\in C_c^\infty(\mathbb R)\) there exists a measurable kernel
\(K_{\beta,m}(t,s)\) such that
\begin{equation}\label{eq:5.8}
w(t)=\int_{\mathbb R} K_{\beta,m}(t,s)\, \big[ Q_{\beta,m} w\big](s)\,\mathrm{d}s.
\end{equation}
Let \(\delta=|t-s|\) and \(u=(t+s)/2\). Then the kernel satisfies that
\begin{equation}\label{eq:5.9}
|K_{\beta,m}(t,s)| \le \frac{C}{\sqrt{\beta}+m} e^{-c\beta \delta^2},\qquad 0< m\le \sqrt{\beta},
\end{equation}
and, for \(m\ge \sqrt{\beta}\),
\begin{equation}\label{eq:5.10}
|K_{\beta,m}(t,s)| \le \frac{C}{\sqrt{\beta}+m} e^{-c\beta \delta^2}
\Bigl(e^{-c|m-\beta u|\delta}+e^{-c|m+\beta u|\delta}\Bigr),
\end{equation}
where the constants \(C,c>0\) are independent of \(\beta,m\).
\end{lemma}

\begin{lemma}\label{lem:bandsum}
Let \(\gamma=2/d\), \(a_0=1-\gamma\), and \(\beta\ge 1\). Uniformly in \(M\ge 0\)
and \(\delta>0\),

\begin{equation}
e^{-c\beta\delta^2}
\left[
\beta^{-1/2}\sum_{0\le j\le 2\sqrt{\beta}}(1+j)^{a_0}
+\sum_{j>2\sqrt{\beta}} j^{-\gamma} e^{-c|j-M|\delta}
\right]
\le C\delta^{-a_0}.
\label{eq:sum}
\end{equation}
\end{lemma}

\vspace{2em}

\subsection{Proof of Theorem \ref{thm:main}}

The proofs of Lemma \ref{lem:kernel} and Lemma \ref{lem:bandsum} are deferred to the end of this section. With the two lemmas in hand, we now prove 
the  main theorem.

\begin{proof}
We define the cluster operator 
\begin{equation}\label{def:cluster}
E_jf=\chi_{[j,j+1)} f=\sum_{\{k:\,\lambda^2_k\in I_j\} }\mathcal{H}_k f
\end{equation} where $I_j := [j^2,(j+1)^2)$ and $\mathcal{H}_k$ is given by \eqref{def:Hk}.
Then, for $w\in L^2(Y)$
\[
w = \sum_j E_j w = \sum_j w_j= \sum_j \sum_{\{k:\,\lambda^2_k\in I_j\} }\mathcal{H}_k w.
\]

Applying Lemma \ref{lem:kernel}, $\mathcal{H}_k w$ admits the following representation.
\[\mathcal{H}_k  w=
\int_{\mathbb R} K_{\beta,\lambda_k}(t,s)\, e^{\frac{\tau t^2}{4}} P e^{-\frac{\tau t^2}{4}} \mathcal{H}_k w\,\mathrm{d}s
\]
Note that $\mathcal{H}_k$ is an angular operator, and it commutes with the radial operator.
\[
w= \sum_j \sum_{\lambda_k^2 \in I_j} \int_{\mathbb R} K_{\beta,\lambda_k}(t,s)\, \mathcal{H}_k F\,\mathrm{d}s,
\]
We next separate the radial and angular components. The estimates for the angular part are mainly based on Sogge's estimates \textup{Proposition 2.1 of \cite{Sog88}} and \(L^2\) orthogonality.
\begingroup
\allowdisplaybreaks
\begin{align}
\|w\|_{L^q(dh)}
&\le \sum_j \biggl\|\sum_{\lambda_k^2 \in I_j} \int_{\mathbb R} K_{\beta,\lambda_k}(t,s)\,\mathcal{H}_k F\,\mathrm{d}s\biggr\|_{L^q(dh)} \nonumber\\
&\le \sum_j (1+j)^{\frac12-\frac1d} \biggl\|\sum_{\lambda_k^2\in I_j} \int_{\mathbb R} K_{\beta,\lambda_k}(t,s)\,\mathcal{H}_k F\,\mathrm{d}s\biggr\|_{L^2(dh)} \nonumber\\
&\le \sum_j (1+j)^{\frac12-\frac1d} \int_{\mathbb R} \biggl\|\sum_{\lambda_k^2\in I_j} K_{\beta,\lambda_k}(t,s)\,\mathcal{H}_k F\biggr\|_{L^2(dh)}\,\mathrm{d}s \nonumber\\
&= \sum_j (1+j)^{\frac12-\frac1d}
\int_{\mathbb R}
\biggl(\int_{ Y} \Bigl|\sum_{\lambda_k^2\in I_j} K_{\beta,\lambda_k}(t,s)\mathcal{H}_k F\Bigr|^2 \mathrm{d}h\biggr)^{\frac12} \mathrm{d}s \nonumber\\
&= \sum_j (1+j)^{\frac12-\frac1d}
\int_{\mathbb R}
\biggl(\int_{ Y}
\Bigl(\sum_{\lambda_k^2\in I_j} K_{\beta,\lambda_k} \mathcal{H}_k F\Bigr)
\,\overline{\Bigl(\sum_{\lambda_{k'}^2\in I_j} K_{\beta,\lambda_{k'}} \mathcal{H}_{k'} F\Bigr)}
\,\mathrm{d}h\biggr)^{\frac12} \mathrm{d}s \nonumber\\
&\le \sum_j (1+j)^{\frac12-\frac1d}
\int_{\mathbb R}\sup_{\lambda_k^2\in I_j}|K_{\beta,\lambda_k}|
\biggl(\int_{ Y}\sum_{\lambda_k^2\in I_j} |\mathcal{H}_k F|^2 \mathrm{d}h\biggr)^{\frac12} \mathrm{d}s \nonumber\\
&= \sum_j (1+j)^{\frac12-\frac1d}
\int_{\mathbb R}\sup_{\lambda_k^2\in I_j}|K_{\beta,\lambda_k}|\,\|E_j F\|_{L^2(dh)}\,\mathrm{d}s \nonumber\\
&\le \sum_j (1+j)^{1-\frac2d}
\int_{\mathbb R}\sup_{\lambda_k^2\in I_j}|K_{\beta,\lambda_k}|\,\|F\|_{L^p(dh)}\,\mathrm{d}s,
\label{eq:w_estimate}
\end{align}
\endgroup
For \(\lambda_k\in [j,j+1)\), set \(\delta=|t-s|\), \(u=\frac{t+s}{2}\), \(M=\beta|u|\).
Let \(c_0\) denote the constant appearing in \eqref{eq:5.9} and \eqref{eq:5.10}.

We claim that for every \(j\),
\begin{equation}
\sup_{\lambda_k^2\in I_j}\bigl|K_{\beta,\lambda_k}(t,s)\bigr|
\le C\frac{e^{-c\beta\delta^2}}{\sqrt{\beta}+j},
\label{eq:4.16}
\end{equation}
Moreover, for \(j>2\sqrt{\beta}\), we have the stronger bound
\begin{equation}
\sup_{\lambda_k^2\in I_j}\bigl|K_{\beta,\lambda_k}(t,s)\bigr|
\le C j^{-1} e^{-c\beta\delta^2} e^{-c|j-M|\delta}.
\end{equation}

First, \eqref{eq:5.9} gives
\[
\sup_{\lambda_k^2\in I_j}\bigl|K_{\beta,\lambda_k}(t,s)\bigr|
\le C\frac{e^{-c_0\beta\delta^2}}{\sqrt{\beta}+\lambda_k}
\le C\frac{e^{-c\beta\delta^2}}{\sqrt{\beta}+j},
\qquad \bigl(c=\tfrac{c_0}{2}\bigr).
\]
This finishes the proof of \eqref{eq:4.16}.
\vspace{2em}
When \(j>2\sqrt{\beta}\), we have \(\lambda_k\ge j>2\sqrt{\beta}\), so we may apply
the estimate \eqref{eq:5.10}:
\[
\begin{aligned}
  \sup_{\lambda_k^2\in I_j}\bigl|K_{\beta,\lambda_k}(t,s)\bigr|
  &\le \sup_{\ell\in I_j} \frac{C}{\sqrt{\beta}+\lambda_k}
     \, e^{-c_0\beta\delta^2}
     \bigl(e^{-c_0|\lambda_k-\beta u|\delta}+e^{-c_0|\lambda_k+\beta u|\delta}\bigr) \\
  &\le 2C j^{-1} e^{-c_0\beta\delta^2}
     \sup_{\ell\in I_j} e^{-c_0|\lambda_k-M|\delta}.
\end{aligned}
\]

Since \(j\le \lambda_k\le j+1\), the reverse triangle inequality gives
\[
-|\lambda_k-M|\le 1-|j-M|.
\]
Therefore,
\[
\begin{aligned}
  \sup_{\lambda_k^2\in I_j}\bigl|K_{\beta,\lambda_k}(t,s)\bigr|
  &\le 2C j^{-1} e^{-c_0\beta\delta^2}
     e^{c_0\delta(1-|j-M|)} \\
  &\lesssim j^{-1} e^{-\frac{c_0}{2}\beta\delta^2}
     e^{-\frac{c_0}{2}|j-M|\delta} \\
  &= j^{-1} e^{-c\beta\delta^2} e^{-c|j-M|\delta}.
\end{aligned}
\]

This establishes the two pointwise estimates for the kernel,
which gives
\begin{equation}
\begin{aligned}
\|w\|_{L^q(dh)}
&\le \sum_j (1+j)^{1-\frac{2}{d}}
\int_{\mathbb R}\sup_{\lambda_k^2\in I_j}|K_{\beta,\lambda_k}|\,\|F\|_{L^p(dh)}\,\mathrm{d}s \\
&\le \sum_{j\le 2\sqrt{\beta}} (1+j)^{1-\frac{2}{d}}
\int_{\mathbb R}\sup_{\lambda_k^2\in I_j}|K_{\beta,\lambda_k}|\,\|F\|_{L^p(dh)}\,\mathrm{d}s \\
&\quad+\sum_{j> 2\sqrt{\beta}} (1+j)^{1-\frac{2}{d}}
\int_{\mathbb R}\sup_{\lambda_k^2\in I_j}|K_{\beta,\lambda_k}|\,\|F\|_{L^p(dh)}\,\mathrm{d}s \\
&\lesssim \int_{\mathbb R}
\Biggl(
\sum_{j\le 2\sqrt{\beta}} (1+j)^{1-\frac{2}{d}} \frac{e^{-c\beta \delta^2}}{\sqrt{\beta}+j}
+\sum_{j>2\sqrt{\beta}} (1+j)^{1-\frac{2}{d}} j^{-1} e^{-c\beta \delta^2} e^{-c|j-M|\delta}
\Biggr)\\
&\qquad \times
\|F\|_{L^p(dh)}\,\mathrm{d}s \\
&\lesssim \int_{\mathbb R} e^{-c\beta \delta^2}
\Biggl(
\sum_{j\le 2\sqrt{\beta}} (1+j)^{1-\frac{2}{d}} \beta^{-\frac12}
+\sum_{j>2\sqrt{\beta}} j^{-\frac{2}{d}} e^{-c|j-M|\delta}
\Biggr)  \|F\|_{L^p(dh)}\,\mathrm{d}s
\end{aligned}
\label{eq:w_estimate1}
\end{equation}

Then from \eqref{eq:w_estimate1} and \eqref{eq:sum},
\begin{align*}
\|w\|_{L^q(\mathrm{d}t\,\mathrm{d}h)}
&\lesssim
\biggl\|
\int_{\mathbb R}|t-s|^{-1+\frac{2}{d}}\|F\|_{L^p(\mathrm{d}h)}\,\mathrm{d}s
\biggr\|_{L^q(\mathrm{d}t)}
\lesssim \|F\|_{L^p(\mathrm{d}t\,\mathrm{d}h)}.
\end{align*}
The last step uses the Hardy--Littlewood--Sobolev inequality \cite{LL01}.

\end{proof}

This finishes the proof of our main theorem. In what follows, we prove Lemma \ref{lem:kernel} and Lemma \ref{lem:bandsum}.

\subsection{Proof of Lemma \ref{lem:kernel}}

\begin{proof}
One may verify directly that
\[
K_0(t,s)=\frac{e^{h_\beta(t)-h_\beta(s)-m|t-s|}}{2m}
\]
satisfies
\[
w(t)=\int_{\mathbb R} K_0(t,s)\,Q_{\beta,m} w(s)\,\mathrm{d}s.
\]

Indeed,
\begin{align*}
&\int_{\mathbb R} K_0(t,s)\,Q_{\beta,m} w(s)\,\mathrm{d}s\\
&=\frac{e^{h_\beta(t)-mt}}{2m}\int_{-\infty}^t e^{ms}\bigl(-\partial_s^2+m^2\bigr)f(s)\,\mathrm{d}s
+\frac{e^{h_\beta(t)+mt}}{2m}\int_{t}^{+\infty} e^{-ms}\bigl(-\partial_s^2+m^2\bigr)f(s)\,\mathrm{d}s\\
&=e^{h_\beta(t)}f(t)=w(t),
\end{align*}
where
\[
f(s)=e^{-h_\beta(s)}w(s).
\]

On the other hand, it is not difficult to see that
\[
\int_{\mathbb R} e^{\pm ms}\bigl(-\partial_s^2+m^2\bigr)f(s)\,\mathrm{d}s = 0,
\]
that is,
\[
\int_{\mathbb R} e^{-h_\beta(s)\pm ms}\, Q_{\beta,m} w(s)\,\mathrm{d}s = 0.
\]

Thus we obtain
\[
K_R(t,s)=K_0(t,s)-\frac{1}{2m}e^{h_\beta(t)-mt-h_\beta(s)+ms},
\]
\[
K_L(t,s)=K_0(t,s)-\frac{1}{2m}e^{h_\beta(t)+mt-h_\beta(s)-ms},
\]
both of which are left inverses for \(Q_{\beta,m}\).
More precisely,
\[
K_R(t,s)=
\begin{cases}
0, & s<t,\\[4pt]
-m^{-1}e^{h_\beta(t)-h_\beta(s)}\sinh\bigl(m(s-t)\bigr), & s>t,
\end{cases}
\]
\[
K_L(t,s)=
\begin{cases}
m^{-1}e^{h_\beta(t)-h_\beta(s)}\sinh\bigl(m(s-t)\bigr), & s<t,\\[4pt]
0, & s>t.
\end{cases}
\]

We now select a suitable left inverse kernel to establish estimates \eqref{eq:5.9}
and \eqref{eq:5.10}.

\medskip
\noindent\textbf{Case 1 $m\ge \sqrt{\beta}$.}
Set \(\delta=|t-s|\), \(u=\frac{t+s}{2}\). We first record a simple auxiliary
inequality: if \(A\ge \frac{\beta \delta}{2}\) and \(0<c_0<\frac13\), then
\begin{equation}
e^{-A\delta}\le e^{-c_0\beta\delta^2}e^{-c_0 A\delta},
\label{eq:basic_ineq}
\end{equation}
since \((1-c_0)A\ge c_0\beta\delta\).

We choose the left inverse according to three regimes for \(t\):
\begin{itemize}
\item (Regime A) When \(t<-\dfrac{m}{\beta}\), take \(K_L\) as the left inverse of \(Q_{\beta,m}\).
\item (Regime B) When \(|t|\le \dfrac{m}{\beta}\), take \(K_0\) as the left inverse of \(Q_{\beta,m}\).
\item (Regime C) When \(t>\dfrac{m}{\beta}\), take \(K_R\) as the left inverse of \(Q_{\beta,m}\).
\end{itemize}

\medskip
\noindent\textbf{(Regime B): Case $|t|\le \dfrac{m}{\beta}$.}

\smallskip
\noindent\textit{(Case B1): $t>s$.}
\begin{equation}
\begin{aligned}
h_\beta(t)-h_\beta(s)-m|t-s|
&=(\beta u-m)(t-s)\\
&=-(m-\beta u)\delta,
\end{aligned}
\label{eq:case_B1_tgreater_s}
\end{equation}
Note that
\begin{equation}
m-\beta u
=m-\beta t+\frac{\beta\delta}{2}
\ge \frac{\beta\delta}{2},
\label{eq:lower_m_beta_u_B1}
\end{equation}
Hence, applying \eqref{eq:basic_ineq},
\begin{equation}
\begin{aligned}
K_0=\frac{1}{2m}e^{-(m-\beta u)\delta}
&\le \frac{1}{2m}e^{-c_0(m-\beta u)\delta}e^{-c_0\beta\delta^2}\\
&=\frac{1}{2m}e^{-c_0|m-\beta u|\delta}e^{-c_0\beta\delta^2}.
\end{aligned}
\label{eq:K0_est_B1}
\end{equation}

\smallskip
\noindent\textit{(Case B2): $s>t$.}
\begin{equation}
\begin{aligned}
h_\beta(t)-h_\beta(s)-m|t-s|
&=-(\beta u+m)(s-t)\\
&=-(m+\beta u)\delta,
\end{aligned}
\label{eq:case_B2_sgreater_t}
\end{equation}
Observe that
\begin{equation}
m+\beta u
=m+\beta\frac{t+s}{2}
=m+\beta t+\frac{\beta\delta}{2}
\ge \frac{\beta\delta}{2},
\label{eq:lower_m_plus_beta_u_B2}
\end{equation}
Thus, by \eqref{eq:basic_ineq},
\begin{equation}
K_0=\frac{1}{2m}e^{-(m+\beta u)\delta}
\le \frac{1}{2m}e^{-c_0|m+\beta u|\delta}e^{-c_0\beta\delta^2},
\label{eq:K0_est_B2}
\end{equation}
Recall that \(m\ge\sqrt{\beta}\). Combining Cases B1 and B2,
\begin{equation}
|K_0|\le \frac{C}{m+\sqrt{\beta}}e^{-c_0\beta\delta^2}
\Bigl(e^{-c_0|m-\beta u|\delta}+e^{-c_0|m+\beta u|\delta}\Bigr).
\label{eq:K0_final_B}
\end{equation}

\medskip
\noindent\textbf{(Regime A): Case $t<-\dfrac{m}{\beta}$, $s<t$.}
We use the left-inverse kernel \(K_L\).
\begin{equation}
\begin{aligned}
K_L&=\frac{1}{2m}\Bigl(e^{h_\beta(t)-h_\beta(s)-mt+ms}-e^{h_\beta(t)-h_\beta(s)+mt-ms}\Bigr)\\
&=\frac{1}{2m}\Bigl(e^{\beta u(t-s)-m(t-s)}-e^{\beta u(t-s)+m(t-s)}\Bigr)\\
&=\frac{1}{2m}\Bigl(e^{-(m-\beta u)\delta}-e^{-(-m-\beta u)\delta}\Bigr)
=\frac1{2m}\Bigl(e^{-|m-\beta u|\delta}-e^{-|m+\beta u|\delta}\Bigr).
\end{aligned}
\label{eq:KL_def_A}
\end{equation}

Since \(s<t\), we have \(u=t-\frac{\delta}{2}<-\frac{m}{\beta}-\frac{\delta}{2}\), so
\(\beta u+m<-\frac{\beta\delta}{2}\). Hence \(-\beta u-m>\frac{\beta\delta}{2}>0\),
and \(m-\beta u>-\beta u-m>0\). Therefore \(m-\beta u=|m-\beta u|\),
\(-m-\beta u=|m+\beta u|\), and
\begin{equation}
|m-\beta u|=m-\beta u> -m-\beta u=|m+\beta u|\ge \frac{\beta\delta}{2}.
\label{eq:abs_estimate_A}
\end{equation}

Applying inequality \eqref{eq:basic_ineq} to both exponential terms,
\begin{equation}
|K_L|\lesssim \frac{C}{m+\sqrt{\beta}}e^{-c_0\beta\delta^2}
\Bigl(e^{-c_0|m+\beta u|\delta}+e^{-c_0|m-\beta u|\delta}\Bigr).
\label{eq:KL_est_A}
\end{equation}

\medskip
\noindent\textbf{(Regime C): Case $t>\dfrac{m}{\beta}$, $s>t$.}
We use the left-inverse kernel \(K_R\).

For \(s>t\), \(u=t+\frac{\delta}{2}>\frac{m}{\beta}+\frac{\delta}{2}\), so
\(\beta u-m>\frac{\beta\delta}{2}>0\).

We have
\begin{equation}
\frac{\beta\delta}{2}<|m-\beta u|=\beta u-m<m+\beta u=|m+\beta u|.
\label{eq:abs_estimate_C}
\end{equation}
Thus
\begin{equation}
\begin{aligned}
K_R&=\frac1{2m}\Bigl(e^{h_\beta(t)+mt-h_\beta(s)-ms}-e^{h_\beta(t)-mt-h_\beta(s)+ms}\Bigr)\\
&=\frac1{2m}\Bigl(e^{\beta u(t-s)+m(t-s)}-e^{\beta u(t-s)-m(t-s)}\Bigr)\\
&=\frac1{2m}\Bigl(e^{-(m+\beta u)(s-t)}-e^{-(\beta u-m)(s-t)}\Bigr)
=\frac1{2m}\Bigl(e^{-|m+\beta u|\delta}-e^{-|m-\beta u|\delta}\Bigr).
\end{aligned}
\label{eq:KR_def_C}
\end{equation}

Applying \eqref{eq:basic_ineq} once more,
\begin{equation}
|K_R|\lesssim \frac{C\,e^{-c_0\beta\delta^2}}{m+\sqrt{\beta}}
\Bigl(e^{-c_0|m+\beta u|\delta}+e^{-c_0|m-\beta u|\delta}\Bigr).
\label{eq:KR_est_C}
\end{equation}

\medskip
Collecting Regimes A, B and C, for \(m\ge\sqrt{\beta}\) we obtain a measurable
left-inverse kernel \(K_{\beta,m}\) satisfying
\begin{equation}
|K_{\beta,m}(t,s)|\le \frac{C}{\sqrt{\beta}+m}e^{-c\beta\delta^2}
\Bigl(e^{-c|m-\beta u|\delta}+e^{-c|m+\beta u|\delta}\Bigr),
\label{eq:final_kernel_est}
\end{equation}
which is exactly estimate \eqref{eq:5.10}.

\medskip
\noindent\textbf{Case 2 $0<m\le \sqrt{\beta}$.}

\smallskip
\noindent\textit{(Case 2.1): $t\ge 0$.}
We select \(K_R\), which vanishes for \(s<t\) and is non-zero only when \(s>t\).
Set
\[
x=\sqrt{\beta}\,t\ge 0,\quad \rho=\sqrt{\beta}(s-t)>0,\quad \nu=\frac{m}{\sqrt{\beta}}\in(0,1),
\]
We compute
\[
h_\beta(t)-h_\beta(s)=\beta\Bigl(\frac{t^2-s^2}{2}\Bigr)
=-x\rho-\frac{\rho^2}{2},
\]
Then
\begin{equation}
\begin{aligned}
|K_R|&= m^{-1} e^{h_\beta(t)-h_\beta(s)} \sinh\big(m(s-t)\big) \\
&=\beta^{-\frac12} e^{-x\rho-\frac{\rho^2}{2}} \beta^{\frac12}m^{-1}\sinh(\nu\rho) \\
&=\beta^{-\frac12} e^{-x\rho-\frac{\rho^2}{2}}\frac{\sinh(\nu\rho)}{\nu\rho}\ \rho,
\end{aligned}
\label{eq:KR_lowm_def}
\end{equation}
Recall the elementary inequality
\[
\frac{\sinh x}{x}\le e^{x},\quad \forall x>0,
\]
Hence
\[
|K_R|\le \beta^{-\frac12}\rho\, e^{-x\rho-\frac{\rho^2}{2}+\nu\rho},
\]
Observe the bound
\[
\nu\rho\le \frac{\rho^2}{4}+\nu^2 \le \frac{\rho^2}{4}+1,
\]
Substituting this estimate,
\begin{equation}
\begin{aligned}
|K_R|&\le \beta^{-\frac12}\rho\, e^{-x\rho-\frac{\rho^2}{4}+1}
\le C\,\beta^{-\frac12}\rho\, e^{-\frac{\rho^2}{4}}
\le C\,\beta^{-\frac12} e^{-\frac{\rho^2}{8}},
\end{aligned}
\label{eq:KR_lowm_est}
\end{equation}
Since \(\rho=\sqrt{\beta}\delta\), we have \(e^{-\rho^2/8}=e^{-\beta\delta^2/8}\). Therefore
\begin{equation}
|K_R|\le\frac{C}{m+\sqrt{\beta}}\,e^{-c\beta\delta^2}.
\label{eq:KR_lowm_final}
\end{equation}

\smallskip
\noindent\textit{(Case 2.2): $t< 0$.}
We take \(K_L\), which vanishes for \(s>t\) and is non-zero only for \(s<t\).
Set
\[
x=-\sqrt{\beta}\,t\ge 0,\quad \rho=\sqrt{\beta}(t-s)>0,\quad \nu=\frac{m}{\sqrt{\beta}}\in(0,1).
\]

Using the explicit form of \(K_L\),
\begin{equation}
\begin{aligned}
|K_L|&=\Bigl| m^{-1} e^{h_\beta(t)-h_\beta(s)} \sinh\big(m(t-s)\big) \Bigr| \\
&=\Bigl| m^{-1} e^{-x\rho-\frac{\rho^2}{2}} \sinh(\nu\rho) \Bigr|
=\Bigl| m^{-1} e^{-x\rho-\frac{\rho^2}{2}}  \nu\rho \frac{\sinh(\nu\rho)}{\nu\rho}\Bigr| \\
&=\Bigl| \beta^{-1/2} \rho\, e^{-x\rho-\frac{\rho^2}{2}}  \frac{\sinh(\nu\rho)}{\nu\rho}\Bigr|
\le \beta^{-1/2}\rho\, e^{-x\rho-\frac{\rho^2}{2}+\nu\rho}\\
&\le \beta^{-1/2}\rho\, e^{-x\rho-\frac{\rho^2}{4}+1}
\le C\beta^{-1/2} e^{-\rho^2/8}
\lesssim \frac{C}{m+\sqrt{\beta}}\, e^{-c\beta\delta^2}.
\end{aligned}
\label{eq:KL_lowm_est}
\end{equation}

Combining Case 2.1 and Case 2.2, we obtain the kernel bound \eqref{eq:5.9} for
\(0<m\le\sqrt{\beta}\). This completes the proof of Lemma \ref{lem:kernel}.
\end{proof}

\subsection{Proof of Lemma~\ref{lem:bandsum}}

\begin{proof}
We treat the two sums separately.

\medskip
\noindent\textbf{First term: $\displaystyle e^{-c\beta\delta^2}\beta^{-1/2}\sum_{0\le j\le 2\sqrt{\beta}}(1+j)^{a_0}$.}

Since \(j\le 2\sqrt{\beta}\), we have \((1+j)^{a_0}\le C(1+2\sqrt{\beta})^{a_0}\). The
number of summands is \(O(\sqrt{\beta})\), so
\[
\beta^{-1/2}\sum_{0\le j\le 2\sqrt{\beta}}(1+j)^{a_0}
\le C\beta^{-1/2}\cdot \sqrt{\beta}\,(2\sqrt{\beta})^{a_0}
= C\beta^{\frac{a_0}{2}}.
\]

Thus
\[
e^{-c\beta\delta^2}\beta^{-1/2}\sum_{0\le j\le 2\sqrt{\beta}}(1+j)^{a_0}
\le C\, e^{-c\beta\delta^2}\beta^{\frac{a_0}{2}}.
\]

Set \(y=\sqrt{\beta}\delta\), so \(\beta^{\frac{a_0}{2}}=y^{a_0}\delta^{-a_0}\). Then
\[
e^{-c\beta\delta^2}\beta^{\frac{a_0}{2}}= e^{-c y^2} y^{a_0} \,\delta^{-a_0}.
\]

Because \(\sup_{y\ge 0} y^{a_0} e^{-c y^2}<+\infty\), we obtain
\[
e^{-c\beta\delta^2}\beta^{-1/2}\sum_{0\le j\le 2\sqrt{\beta}}(1+j)^{a_0}
\le C\delta^{-a_0}.
\]

\medskip
\noindent\textbf{Second term: $\displaystyle e^{-c\beta\delta^2}\sum_{j>2\sqrt{\beta}} j^{-\gamma}e^{-c|j-M|\delta}$.}

We split into two cases: \(\delta\ge 1\) and \(0<\delta<1\).

\smallskip
\noindent\textit{Case 1: $\delta\ge 1$.}
\begin{align*}
e^{-c\beta\delta^2}\sum_{j>2\sqrt{\beta}} j^{-\gamma}e^{-c|j-M|\delta}
&\le e^{-c\beta\delta^2}\sum_{j=0}^{+\infty}e^{-c|j-M|\delta}\\
&\le e^{-c\beta\delta^2}\Bigl(\sum_{j=0}^{\lfloor M\rfloor}e^{-c(M-j)\delta}
+\sum_{j=\lfloor M\rfloor+1}^{+\infty}e^{-c(j-M)\delta}\Bigr)\\
&\le C\delta^{-a_0},
\end{align*}
where we used \(a_0=1-\gamma>0\), hence \(\delta^{-1}\le \delta^{-a_0}\) for
\(\delta\ge 1\).

\smallskip
\noindent\textit{Case 2: $0<\delta<1$.}
We further subdivide according to the size of \(M\).

\paragraph{(i) $M\le \dfrac{2}{\delta}$.}
Split the sum at \(j=\dfrac{4}{\delta}\):
\begin{align*}
\sum_{j>2\sqrt{\beta}} j^{-\gamma}e^{-c|j-M|\delta}
=\sum_{\substack{j>2\sqrt{\beta}\\j\le \frac4\delta}} j^{-\gamma}e^{-c|j-M|\delta}
+\sum_{j>\frac4\delta} j^{-\gamma}e^{-c|j-M|\delta}.
\end{align*}

\[
\sum_{\substack{j>2\sqrt{\beta}\\j\le \frac4\delta}} j^{-\gamma}e^{-c|j-M|\delta}
\lesssim \sum_{\substack{j>2\sqrt{\beta}\\j\le \frac4\delta}} j^{-\gamma}
\lesssim \delta^{\gamma-1}=\delta^{-a_0}.
\]

For the second sub-sum, \(j>\frac4\delta\) and \(M\le \frac2\delta\), so
\(|j-M|\ge j-M\ge \frac j2\):
\[
\sum_{j>\frac4\delta} j^{-\gamma}e^{-c|j-M|\delta}
\le \sum_{j>\frac4\delta}\Bigl(\frac4\delta\Bigr)^{-\gamma} e^{-c\frac12 j\delta}
\lesssim \delta^{\gamma}\sum_{j>\frac4\delta}e^{-\frac c2 j\delta}
\lesssim \delta^{\gamma}\cdot \delta^{-1}=\delta^{-a_0}.
\]

Multiplying by \(e^{-c\beta\delta^2}\le 1\), we get
\[
e^{-c\beta\delta^2}\sum_{j>2\sqrt{\beta}} j^{-\gamma}e^{-c|j-M|\delta}\le C\delta^{-a_0}.
\]

\paragraph{(ii) $M>\dfrac{2}{\delta}$.}
We partition the summation range \(j>2\sqrt{\beta}\) into three parts:
\[
\sum_{j>2\sqrt{\beta}}=\sum_{\substack{j>2\sqrt{\beta}\\j<\frac M2}}
+\sum_{\substack{j>2\sqrt{\beta}\\|j-M|\le \frac M2}}
+\sum_{\substack{j>2\sqrt{\beta}\\j>\frac32 M}}.
\]

 Part 1: \(2\sqrt{\beta}<j<\frac M2\). Here \(|j-M|\ge \frac M2\), so
\begin{align*}
e^{-c\beta\delta^2}\sum_{\substack{j>2\sqrt{\beta}\\j<\frac M2}} j^{-\gamma}e^{-c|j-M|\delta}
&\lesssim e^{-c\beta\delta^2}e^{-c\frac M2\delta}\sum_{\substack{j>2\sqrt{\beta}\\j<\frac M2}} j^{-\gamma}\\
&\lesssim e^{-c\beta\delta^2}e^{-c\frac M2\delta} M^{1-\gamma}\\
&\lesssim e^{-c\beta\delta^2}\,\delta M\, e^{-\frac c2 M\delta} M^{-\gamma}\delta^{-1}\\
&\lesssim \delta^{-a_0}.
\end{align*}

 Part 2: \(|j-M|\le \frac M2\). In this region \(j\simeq M\), so \(j^{-\gamma}\lesssim M^{-\gamma}\):
\begin{align*}
e^{-c\beta\delta^2}\sum_{\substack{j>2\sqrt{\beta}\\|j-M|\le \frac M2}} j^{-\gamma}e^{-c|j-M|\delta}
&\lesssim e^{-c\beta\delta^2}M^{-\gamma}
\biggl(\sum_{\frac M2\le j\le M}e^{-c(M-j)\delta}
+\sum_{M<j\le \frac32 M}e^{-c(j-M)\delta}\biggr)\\
&\lesssim e^{-c\beta\delta^2}M^{-\gamma} (1+\delta)^{-1},
\end{align*}
since \(M>\frac2\delta\), \(M^{-\gamma}\lesssim \delta^{\gamma}\), hence
\[
e^{-c\beta\delta^2}M^{-\gamma}(1+\delta)^{-1}
\lesssim e^{-c\beta\delta^2}\delta^{\gamma}(1+\delta)^{-1}
\le C\delta^{-a_0}.
\]

 Part 3: \(j>\frac32 M\). Then \(|j-M|\ge \frac j3\), so
\begin{align*}
e^{-c\beta\delta^2}\sum_{\substack{j>2\sqrt{\beta}\\j>\frac32 M}} j^{-\gamma}e^{-c|j-M|\delta}
&\lesssim e^{-c\beta\delta^2}\sum_{j>\frac32 M} j^{-\gamma}e^{-c\frac13 j\delta}\\
&\lesssim e^{-c\beta\delta^2} M^{-\gamma}\sum_{j>\frac32 M} e^{-\frac c3 j\delta}\\
&\lesssim e^{-c\beta\delta^2} M^{-\gamma}\cdot \delta^{-1},
\end{align*}
again \(M>\frac2\delta\) yields \(M^{-\gamma}\lesssim \delta^{\gamma}\), so
\[
e^{-c\beta\delta^2} M^{-\gamma}\delta^{-1}
\lesssim e^{-c\beta\delta^2}\delta^{\gamma}\delta^{-1}
\le C\delta^{-a_0}.
\]

Combining the three sub-parts for \(M>\frac2\delta\), we obtain
\[
e^{-c\beta\delta^2}\sum_{j>2\sqrt{\beta}} j^{-\gamma}e^{-c|j-M|\delta}\le C\delta^{-a_0}.
\]

\medskip

We point out that when estimating sums by integrals, the two quantities may differ
by finitely many boundary terms. These boundary terms are uniformly bounded
constants. After multiplying by the factor \(e^{-c\beta\delta^2}\), they are
trivially dominated by \(C\delta^{-a_0}\).

Putting the estimates for the first term and the second term together, inequality
\eqref{eq:sum} holds uniformly in \(\beta\ge 1\), \(M\ge 0\), \(\delta>0\).
\end{proof}

\section{Applications of the Carleman Estimates}
\label{sec:uc_application}
In this section, we present an application of our Carleman estimate. First we take our compact, connected manifold $Y$ without boundary to be $\mathbb{S}^{d-1}$, which yields a Carleman estimate on $\mathbb R^d$.
\begin{equation}\label{eq:E}
\lVert e^{\tau\phi}u\rVert_{L^{q}(\mathbb R^d\setminus\{0\})}
\le C\,\lVert e^{\tau\phi}H_{A,V}u\rVert_{L^{p}(\mathbb R^d\setminus\{0\})},
\qquad u\in C_c^\infty(\mathbb R^d\setminus\{0\}),
\end{equation}

Before proving Theorem~\ref{prob:5.2}, we first present a scaled version of our Carleman estimate.
Such a Carleman estimate allows us to perform unique continuation in multiple directions.
\subsection{Scale invariance of the estimate}

\begin{lemma}\label{lemma5.3}
For every $R>0$ the estimate
\begin{equation}\label{eq:Carleman-R}
\bigl\|e^{\tau \phi_R}u\bigr\|_{L^q(\mathbb R^d)}
\le C\bigl\|e^{\tau \phi_R}H_{A,V}u\bigr\|_{L^p(\mathbb R^d)},
\qquad u\in C_c^\infty(\mathbb R^d\setminus\{0\}),
\end{equation}
holds for all $\tau\ge \tau_0$ with the same constant $C$, where
\begin{equation}\label{eq:phi-R}
\phi_R(x)=\frac14\bigl(\log(|x|/R)\bigr)^2.
\end{equation}
\begin{proof}

The proof of this lemma is straightforward and follows immediately from a scaling transformation applied to equation \eqref{eq:E}.

For $g(y)=f(Ry)$ the homogeneity of \eqref{eq:A-form} gives
$A(Ry)=R^{-1}A(y)$ and $V(Ry)=R^{-2}V(y)$, hence
\begin{equation}\label{eq:homog}
  (H_{A,V}g)(y)=R^{2}(H_{A,V}f)(Ry) .
\end{equation}
Applying \eqref{eq:E} to $g$ and changing variables $x=Ry$ yields
\[
  R^{-d/q}\norm{e^{\tau\phi_R}f}_{L^{q}_x}
  \le C R^{2-d/p}\norm{e^{\tau\phi_R}H_{A,V}f}_{L^{p}_x} .
\]
By \eqref{def:qp-P},
$d/q=(d-2)/2$ and $2-d/p=2-(d+2)/2=-(d-2)/2=-d/q$, so the powers of $R$ on
the two sides coincide and cancel.  This is \eqref{eq:Carleman-R}.  The
cancellation is precisely the reason the exponents in \eqref{def:qp-P} are
the scale-invariant ones.
\end{proof}
\end{lemma}

\subsection{Unique continuation across the unit sphere}
\begin{proof}[Proof of Theorem~\ref{prob:5.2}]
First, assume $\chi\in C_c^\infty(\mathbb R^d\setminus\{0\})$ and 
$f\in W^{2,p}_{\mathrm{loc}}(\mathbb R^d)$.
We now verify that $\chi f$ satisfies \eqref{eq:Carleman-R}. 
It suffices to check that
$e^{\tau\phi_R}\chi f\in L^q(\mathbb R^d)$ and 
$e^{\tau\phi_R}H_{A,V}(\chi f)\in L^p(\mathbb R^d)$.
The first inclusion is trivial. Compute
\begin{equation}
H_{A,V}(\chi f)
=\chi H_{A,V}f + [H_{A,V},\chi]f \\
= 2i\nabla\chi\cdot(i\nabla+A)f - (\Delta\chi)f.
\label{eq:commutator_cutoff}
\end{equation}

Under our assumptions, one readily has 
$e^{\tau\phi_R}H_{A,V}(\chi f)\in L^p(\mathbb R^d)$.
A density argument then shows that $\chi f$ satisfies 
condition \eqref{eq:Carleman-R}.

\smallskip
\noindent\textbf{Step 1 (extend outward): 
$a<|x|<b\ \Rightarrow\ b<|x|<b+\delta$.}
We next prove that $f=0$ on $a<|x|<b$ implies 
$f=0$ on $b<|x|<b+\delta$.
Choose radii $r_1,r_2,r_3$ such that 
$a<r_1<b<b+\delta<r_2<r_3$.
Take a cutoff function $\chi\in C_c^\infty(\mathbb R^d\setminus\{0\})$ 
as illustrated in Figure~\ref{fig:annular_exterior},
\[
\chi(x)=
\begin{cases}
0, & |x|>r_3,\\
1, & b<|x|<r_2,\\
0, & |x|<r_1,
\end{cases}
\]
and define
\[
S_1=\{x\colon b<|x|<b+\delta\},\qquad
S_2=\{x\colon r_2<|x|<r_3\}.
\]

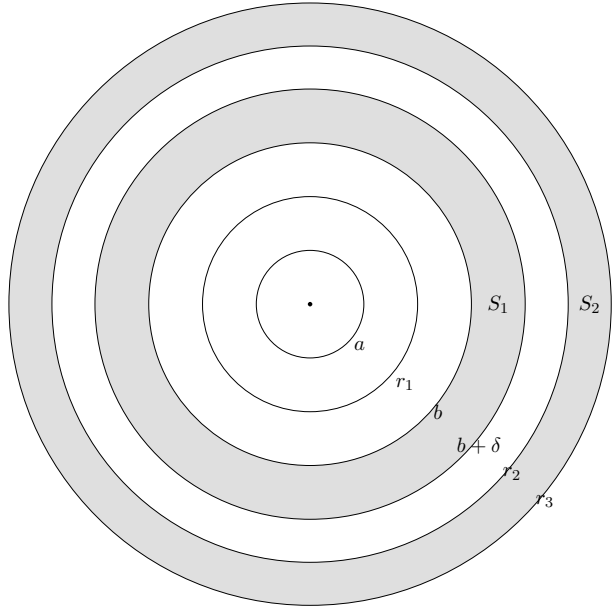
\begin{figure}[H]
\centering
\resizebox{0.55\linewidth}{!}{%
  \begin{tikzpicture}
    \fill (0,0) circle (1.2pt);
    \fill[gray!25, even odd rule] (0,0) circle (4.0) (0,0) circle (3.0);   
    \fill[gray!25, even odd rule] (0,0) circle (5.6) (0,0) circle (4.8);   
    \draw (0,0) circle (1.0);   
    \draw (0,0) circle (2.0);   
    \draw (0,0) circle (3.0);   
    \draw (0,0) circle (4.0);   
    \draw (0,0) circle (4.8);   
    \draw (0,0) circle (5.6);   
    \node at (-40:1.2) {$a$};
    \node at (-40:2.3) {$r_1$};
    \node at (-40:3.1) {$b$};
    \node at (-40:4.1) {$b+\delta$};
    \node at (-40:4.9) {$r_2$};
    \node at (-40:5.7) {$r_3$};
    \node at (0:3.5)   {$S_1$};
    \node at (0:5.2)   {$S_2$};
  \end{tikzpicture}%
}
\caption{Unique continuation from the exterior (Step 1)}
\label{fig:annular_exterior}
\end{figure}

From the Carleman inequality we obtain
\[
\|e^{\tau\phi_R}\chi f\|_{L^q(\mathbb R^d)}
\le C\|e^{\tau\phi_R}H_{A,V}(\chi f)\|_{L^p(\mathbb R^d)}.
\]
Note that $\operatorname{supp}\chi\subset\{x\colon r_1\le|x|\le r_3\}$.
Since $f=0$ for $a<|x|<b$, we have
\[
\operatorname{supp}(\chi f)\subset\{x\colon b\le|x|\le r_3\}.
\]
Moreover,
\[
H_{A,V}(\chi f)=\chi H_{A,V}f + [H_{A,V},\chi]f = [H_{A,V},\chi]f.
\]
Hence $\operatorname{supp}\,H_{A,V}(\chi f)\subset\{x\colon r_2<|x|<r_3\}$, and we get
\[
\|e^{\tau\phi_R}f\|_{L^q(S_1)}
\le C\bigl\|e^{\tau\phi_R}[H_{A,V},\chi]f\bigr\|_{L^p(S_2)}.
\]
Recall that $\phi_R=\frac14\bigl(\log\frac{|x|}{R}\bigr)^2$, which is decreasing as $|x|$ increases when $0<|x|<R$.
Take $R>r_3$. Then
\[
e^{\tau\phi_R(b+\delta)}\|f\|_{L^q(S_1)}
\le C\,e^{\tau\phi_R(r_2)}\bigl\|[H_{A,V},\chi]f\bigr\|_{L^p(S_2)}.
\]
By hypothesis $[H_{A,V},\chi]f\in L^p(S_2)$. Letting $\tau\to\infty$,
\[
\|f\|_{L^q(S_1)}=0 \implies f=0 \quad\text{a.e. in }S_1.
\]
This proves $f\equiv 0$ on $b<|x|<b+\delta$.

\medskip
\noindent\textbf{Step 2 (extend inward): $a-\delta<|x|<b$.}
We now prove that $f=0$ on $a-\delta<|x|<b$.
Choose radii $r_6<r_5<a-\delta<a<r_4<b$, and a radial cutoff function as illustrated in
Figure~\ref{fig:annular_interior},
\[
\chi(x)=
\begin{cases}
0, & |x|>r_4,\\
1, & r_5<|x|<a,\\
0, & |x|<r_6,
\end{cases}
\]
and define
\[
S_3=\{x\colon a-\delta<|x|<a\},\qquad
S_4=\{x\colon r_6<|x|<r_5\}.
\]

\begin{figure}[H]
\centering
\resizebox{0.55\linewidth}{!}{%
  \begin{tikzpicture}
    \fill (0,0) circle (1.2pt);
    \fill[gray!25, even odd rule] (0,0) circle (2.8) (0,0) circle (2.0);   
    \fill[gray!25, even odd rule] (0,0) circle (1.2) (0,0) circle (0.6);   
    \draw (0,0) circle (0.6);   
    \draw (0,0) circle (1.2);   
    \draw (0,0) circle (2.0);   
    \draw (0,0) circle (2.8);   
    \draw (0,0) circle (3.7);   
    \draw (0,0) circle (4.7);   
    \node at (-40:0.8)  {$r_6$};
    \node at (-40:1.4)  {$r_5$};
    \node at (-40:2.2)  {$a-\delta$};
    \node at (-40:3.1)  {$a$};
    \node at (-40:3.9)  {$r_4$};
    \node at (-40:5.0)  {$b$};
    \node at (0:0.95)  {$S_4$};
    \node at (0:2.5)  {$S_3$};
  \end{tikzpicture}%
}
\caption{Unique continuation from the interior (Step 2)}
\label{fig:annular_interior}
\end{figure}
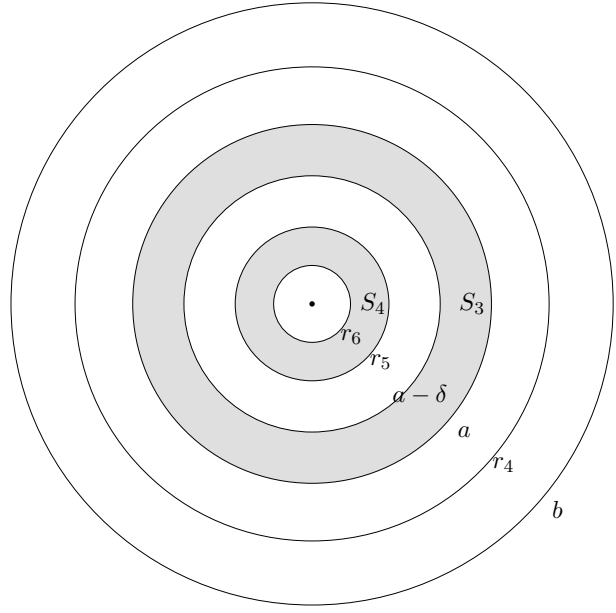

By the same argument as before, we obtain
\[
\|e^{\tau\phi_R} f\|_{L^q(S_3)}
\le C\bigl\|e^{\tau\phi_R}[H_{A,V},\chi]f\bigr\|_{L^p(S_4)}.
\]
Recall that $\phi_R=\frac14\bigl(\log\frac{|x|}{R}\bigr)^2$.
Take $R<r_6$. Then $\phi_R(x)$ is decreasing as $|x|$ decreases for $|x|>R$. Consequently,
\[
e^{\tau\phi_R(a-\delta)}\|f\|_{L^q(S_3)}
\le C\,e^{\tau\phi_R(r_5)}\bigl\|[H_{A,V},\chi]f\bigr\|_{L^p(S_4)}.
\]
Letting $\tau\to\infty$, we get $\|f\|_{L^q(S_3)}=0$, hence $f=0$ a.e.\ in $S_3$.
This proves $f=0$ on $a-\delta<|x|<b$.

This means that the vanishing condition $f=0$ on $D_{a,b}$ can be extended to arbitrary annuli.

We now prove that if $f=0$ a.e.\ on every annulus $0<\frac{1}{R}<|x|<R<+\infty$, then $f=0$ a.e.\ in $\mathbb R^d$.

Set
\[
E \overset{\mathrm{def}}{=} \bigl\{x\in\mathbb R^d\colon f(x)\neq 0\bigr\}.
\]
Suppose for contradiction that $m(E)\neq 0$. Then
\[
0<m(E)\le m\bigl(E\cap B_{\frac1R}\bigr)
+ m\bigl(E\cap (B_{\frac1R}^c\cap B_R)\bigr)
+ m\bigl(E\cap B_R^c\bigr).
\]

Since $f=0$ a.e.\ on every annulus $0<\frac{1}{R}<|x|<R<+\infty$, we have
\[
m\bigl(E\cap (B_{\frac1R}^c\cap B_R)\bigr)=0 \quad \text{for all } R>0.
\]

For sufficiently large $R$,
\[
m\bigl(E\cap B_{\frac1R}\bigr)<\frac{m(E)}{2},
\]
which implies
\[
m\bigl(E\cap B_R^c\bigr)\ge \frac{m(E)}{2}.
\]

Observe that
\[
\begin{aligned}
m\bigl(E\cap B_R^c\bigr)
&= m\left(E\cap \bigcup_{j=R}^{\infty}\big(B_{j+1}\setminus B_j\big)\right)\\
&= m\left(\bigcup_{j=R}^{\infty} \Big(E\cap \big(B_{j+1}\setminus B_j\big)\Big)\right).
\end{aligned}
\]

There must exist some $j_0$ such that $m\big(E\cap (B_{j_0+1}\setminus B_{j_0})\big)>0$.
If not, then
\[
m\left(\bigcup_{j=R}^{\infty} \Big(E\cap \big(B_{j+1}\setminus B_j\big)\Big)\right)
\le \sum_{j=R}^{\infty} m\big(E\cap (B_{j+1}\setminus B_j)\big)=0.
\]

This contradicts
\[
m\left(\bigcup_{j=R}^{\infty} \Big(E\cap \big(B_{j+1}\setminus B_j\big)\Big)\right)\ge \frac{m(E)}{2}.
\]

However, $m\big(E\cap (B_{j_0+1}\setminus B_{j_0})\big)>0$ contradicts our hypothesis.
Hence $m(E)=0$, which completes the proof.
\end{proof}

Thus we have proved Theorem \ref{prob:5.2}. It can be seen from the proof that Lemma \ref{lemma5.3} is very useful, since it allows us to construct suitable weight functions when performing continuation in different directions. Moreover, if $f$ satisfies the assumptions of Theorem \ref{prob:5.2}, then one can conclude that $f=0$ by means of the connectedness argument.

We point out that Lemma \ref{lemma5.3} relies on our choice of exponents. This cannot be achieved for general $L^2$-type Carleman estimates. Our scaled Carleman estimate enables us to cross the unit sphere and construct a locally monotonic weight function $\phi$ for unique continuation, whereas ordinary $L^2$-type Carleman estimates cannot easily perform continuation across the unit sphere.

\subsection{Strong Unique Continuation Property}

\begin{proof}[Proof of Theorem~\ref{theorem1.3}:]

First, let $\rho\in C^\infty(\mathbb R^d)$ be a cut-off function satisfying
\[
\rho=
\begin{cases}
0, & |x|<\tfrac12,\\
1, & |x|>1,
\end{cases}
\qquad 0\le \rho\le 1.
\]
Let $\chi\in C_0^\infty(\mathbb R^d)$ satisfy
\[
\chi=
\begin{cases}
1, & |x|<R_0,\\
0, & |x|>2R_0,
\end{cases}
\]
where $R_0>0$ is taken sufficiently small.

Define
\[
u_j \overset{\mathrm{def}}{=} \rho(jx)u=\rho(jx)\chi f.
\]
Arguing as in the proof of Theorem \ref{prob:5.2}, $u_j$ satisfies the hypotheses required to apply estimate \eqref{eq:E}. Consequently,
\[
\bigl\|e^{\tau\phi} u_j\bigr\|_{L^q(\mathbb R^d)}
\le C\bigl\|e^{\tau\phi} H_{A,V} u_j\bigr\|_{L^p(\mathbb R^d)}.
\]
By definition of the commutator, this gives
\[
\bigl\|e^{\tau\phi} u_j\bigr\|_{L^q(\mathbb R^d)}
\le \bigl\|e^{\tau\phi}\bigl[H_{A,V},\rho(jx)\bigr]u\bigr\|_{L^p(\mathbb R^d)}
+\bigl\|e^{\tau\phi}\rho(jx)\bigl[H_{A,V},\chi\bigr]f\bigr\|_{L^p(\mathbb R^d)}.
\]

Choose $j$ sufficiently large so that $jR_0>1$. The support of the first term on the right-hand side is contained in $\bigl\{x:\tfrac1{2j}<|x|<\tfrac1j\bigr\}$. Since $jR_0>1$, we have $|x|<\tfrac1j<R_0$, hence $\chi=1$ on this set. Therefore
\[
\bigl\|e^{\tau\phi}\bigl[H_{A,V},\rho(jx)\bigr]u\bigr\|_{L^p(\mathbb R^d)}
=\bigl\|e^{\tau\phi}\bigl[H_{A,V},\rho(jx)\bigr]f\bigr\|_{L^p\left(\frac1{2j}<|x|<\frac1j\right)}.
\]

The support of the second term lies in $\{x:R_0<|x|<2R_0\}$. For $|jx|>jR_0>1$, we have $\rho(jx)=1$, so
\[
\bigl\|e^{\tau\phi}\rho(jx)\bigl[H_{A,V},\chi\bigr]f\bigr\|_{L^p(\mathbb R^d)}
=\bigl\|e^{\tau\phi}\bigl[H_{A,V},\chi\bigr]f\bigr\|_{L^p(R_0<|x|<2R_0)}.
\]

Take $R_1<R_0$. Restricting the left-hand side to $B_{R_1}$, we obtain
\[
\begin{aligned}
\bigl\|e^{\tau\phi}\rho(jx)f\bigr\|_{L^q(B_{R_1})}
&\le \bigl\|e^{\tau\phi}\bigl[H_{A,V},\rho(jx)\bigr]f\bigr\|_{L^p\left(\frac1{2j}<|x|<\frac1j\right)}\\
&\quad +\bigl\|e^{\tau\phi}\bigl[H_{A,V},\chi\bigr]f\bigr\|_{L^p(R_0<|x|<2R_0)}.
\end{aligned}
\]

Recall the commutator formula
\[
\bigl[H_{A,V},\rho(jx)\bigr]f
=2i j(\nabla\rho)(jx)\cdot\nabla_A f -j^2(\Delta\rho)(jx)f.
\]
Thus
\[
\begin{aligned}
&\bigl\|e^{\tau\phi}\bigl[H_{A,V},\rho(jx)\bigr]f\bigr\|_{L^p\left(\frac1{2j}<|x|<\frac1j\right)}\\
&\lesssim j\, e^{\frac{\tau}{4}(\log 2j)^2}\bigl\|\nabla_A f\bigr\|_{L^p(|x|<\frac1j)}
+j^2 e^{\frac{\tau}{4}(\log2j)^2}\|f\|_{L^p(|x|<\frac1j)}\\
&\lesssim j^2 e^{\frac{\tau}{4}(\log 2j)^2} e^{-\frac{N}{p}(\log j)^2}.
\end{aligned}
\]

For fixed $\tau$, pick $N>\frac{\tau p}{4}$. Then
\[
\lim_{j\to\infty} j^2 e^{\frac{\tau}{4}(\log 2j)^2} e^{-\frac{N}{p}(\log j)^2}=0.
\]
This yields
\[
\bigl\|e^{\tau\phi}\rho(jx)f\bigr\|_{L^q(B_{R_1})}
\lesssim j^2 e^{\frac{\tau}{4}(\log 2j)^2}e^{-\frac{N}{p}(\log j)^2}
+\bigl\|e^{\tau\phi}\bigl[H_{A,V},\chi\bigr]f\bigr\|_{L^p(R_0<|x|<2R_0)}.
\]

Set $\varphi_k=\bigl|e^{\tau\phi}\rho(kx)f\bigr|^q$. We have $\varphi_k\to |e^{\tau\phi}f|^q$ almost everywhere in $B_{R_1}$ as $k\to\infty$.
Applying Fatou's lemma,
\[
\begin{aligned}
\bigl\|e^{\tau\phi}f\bigr\|_{L^q(B_{R_1})}
&=\left(\int_{B_{R_1}} \bigl|e^{\tau\phi}f\bigr|^q dx\right)^{\frac1q}
=\left(\int_{B_{R_1}} \lim_{j\to\infty}\!\inf_{k\ge j}\varphi_k\,dx\right)^{\frac1q}\\
&\le \lim_{j\to\infty}\!\inf_{k\ge j}\left(\int_{B_{R_1}}\varphi_k\,dx\right)^{\frac1q}
=\lim_{j\to\infty}\!\inf_{k\ge j}
\bigl\|e^{\tau\phi}\rho(kx)f\bigr\|_{L^q(B_{R_1})}\\
&\le\lim_{j\to\infty}\!\inf_{k\ge j}
k^2 e^{\frac{\tau}{4}(\log 2k)^2}e^{-\frac{N}{p}(\log k)^2}
+\bigl\|e^{\tau\phi}\bigl[H_{A,V},\chi\bigr]f\bigr\|_{L^p(R_0<|x|<2R_0)}
\\
&=\bigl\|e^{\tau\phi}\bigl[H_{A,V},\chi\bigr]f\bigr\|_{L^p(R_0<|x|<2R_0)}.
\end{aligned}
\]

We have therefore proved
\[
\bigl\|e^{\tau\phi}f\bigr\|_{L^q(B_{R_1})}
\lesssim \bigl\|e^{\tau\phi}\bigl[H_{A,V},\chi\bigr]f\bigr\|_{L^p(R_0<|x|<2R_0)}.
\]
Equivalently,
\[
\|f\|_{L^q(B_{R_1})}
\lesssim e^{\tau(\phi(R_0)-\phi(R_1))}
\bigl\|\bigl[H_{A,V},\chi\bigr]f\bigr\|_{L^p(R_0<|x|<2R_0)}.
\]

It is straightforward to see that $\bigl[H_{A,V},\chi\bigr]f\in L^p(R_0<|x|<2R_0)$. Letting $\tau\to\infty$, we get $f=0$ a.e.\ in $B_{R_1}$.
 Applying Theorem \ref{prob:5.2}, we conclude $f=0$ a.e.\ in $\mathbb R^d$.

\end{proof}

\end{document}